\documentclass[11pt,a4paper,twoside]{amsart}

\usepackage[T1]{fontenc}
\usepackage[sc]{mathpazo}
\usepackage{amscd}
\usepackage{amsfonts}
\usepackage{amsmath}
\usepackage{amssymb}
\usepackage{amstext}
\usepackage{amsthm}

\usepackage{bbm}
\usepackage{enumerate}
\usepackage{fancyhdr}
\usepackage{setspace}
\usepackage[all,cmtip]{xy}
\usepackage{stmaryrd}
\usepackage{wasysym}

\usepackage{hyperref}

\theoremstyle{definition}
\newtheorem{definition}{Definition}[section]

\theoremstyle{plain}
\newtheorem{thm}[definition]{Theorem}
\newtheorem{lemma}[definition]{Lemma}
\newtheorem{prop}[definition]{Proposition}
\newtheorem{cor}[definition]{Corollary}

\theoremstyle{remark}
\newtheorem{rem}[definition]{Remark}

\newcommand{\U}{\mathcal{U}}

\newcommand{\LIE}{\mathbf{L}}

\newcommand{\LU}{\mathcal{LU}}

\newcommand{\Mstar}{\mathcal{M}_*}
\newcommand{\Mcostar}{\mathcal{M}^*}
\newcommand{\qu}{\mathbb{Q}}
\newcommand{\obb}{\mathbb{O}}

\newcommand{\CL}{\mathcal{C}^{L}}

\renewcommand{\L}{\mathcal{L}}

\newcommand{\Lbox}{\boxtimes_\mathcal{L}}
\newcommand{\Fbox}[2]{\mathrm{F}_{\Lbox}(#1,#2)}

\newcommand{\I}{\mathcal{I}}
\newcommand{\IU}{\mathcal{IU}}

\DeclareMathOperator*{\colim}{colim}

\def \wrt {with respect to }
\newcommand{\R}{\mathbb{R}}

\newcommand{\N}{\mathbb{N}}

\def \fd {finite-dimensional }

\newcommand{\one}{\mathbb{1}}

\newcommand\restr[2]{{
\left.\kern-\nulldelimiterspace 
#1 
\vphantom{\big|} 
\right|_{#2} 
}}

\fancypagestyle{headings}{\pagestyle{fancy}}

\fancypagestyle{headings}{
\fancyhead[LE,RO]{\thepage}
\fancyhead[LO,RE]{}
\fancyhead[CE]{Benjamin B\"ohme}
\fancyhead[CO]{Global model structures for $*$-modules}
\fancyfoot[C]{} 

}

\usepackage[utf8]{inputenc}
\usepackage{color}
\usepackage{geometry} 
\usepackage{setspace}
\author{Benjamin B\"{o}hme}
\address{Department of Mathematical Sciences \\
University of Copenhagen \\
Universitetsparken 5 \\
2100 Copenhagen \\
Denmark}
\email{Boehme@math.ku.dk}
\title{Global model structures for $*$-modules}

\begin{document}

\begin{abstract}
We extend Schwede's work on the unstable global homotopy theory of orthogonal
spaces and $\L$-spaces to the category of $*$-modules (i.e., unstable $S$-modules).
We prove a theorem which transports model structures and their properties from
$\L$-spaces to $*$-modules and show that the resulting global model structure
for $*$-modules is monoidally Quillen equivalent to that of orthogonal spaces.
As a consequence, there are induced Quillen equivalences between the associated
model categories of monoids, which identify equivalent models for the global
homotopy theory of $A_\infty$-spaces.
\end{abstract}

\maketitle
\thispagestyle{empty}

\section{Introduction}
\label{sect:intro}

Global homotopy theory is equivariant homotopy theory \wrt compatible actions of
the family of all compact Lie groups. Unstable global homotopy theory has been
described by Schwede \cite{schwede:GHTv1, schwede:OOUv1} in terms of
various Quillen equivalent model categories that make precise the idea of spaces with
simultaneous group actions, studied up to \emph{global equivalence}, i.e., weak $G$-homotopy
equivalence \wrt all compact Lie groups $G$, in a compatible way.

Two of these models are the categories $\IU$ and $\LU$ of \emph{orthogonal spaces} and
$\L$-\emph{spaces}, respectively. The former is a category of diagram spaces indexed on real
inner product spaces, the latter is the category of spaces equipped with continuous actions of the
topological monoid $\L(1)$ of linear isometric embeddings $\R^\infty \to \R^\infty$ (i.e., the
space of unary operations in the linear isometries operad).

The category of orthogonal spaces admits a ``global'' model structure that is compatible with the
symmetric monoidal structure given by Day convolution. The model structure lifts to the
associated category of monoids and thus models the unstable global homotopy theory of
$A_\infty$ spaces. The same is true for the category of $\L$-\emph{spaces}, up
to a small defect: The operadic box product $\Lbox$ only defines a ``weak'' symmetric monoidal
product that is unital up to global equivalence. The full subcategory $\Mstar \subseteq \LU$ of
$*$-\emph{modules} is spanned by objects such that the unital transformation is an isomorphism,
and is thus symmetric monoidal in the usual sense. It is the unstable analogue of the category of
$S$-modules \cite{EKMM}.

\textbf{Main results:}
Our first main result, Theorem~\ref{thm_B_strict}, establishes a symmetric monoidal model
structure on the category of $*$-modules with weak equivalences the underlying global equivalences of
$\L$-spaces, together with explicit (weak) monoidal Quillen equivalences to both orthogonal spaces
and $\L$-spaces. Theorem~\ref{thm_B_strict} is a direct consequence of a more general ``Transport
Theorem'' (Theorem~\ref{thm_A_strict}) that transports model structures on $\L$-spaces with global
equivalences and well-behaved cofibrations to the category of $*$-modules. It was first proven
in the author's (unpublished) Master's thesis \cite{boehme:masters}. A key step in the proof is
that the category of $*$-modules can be identified with a category of algebras over a monad, which
has previously been used in \cite{EKMM, BCS:THH} to construct non-equivariant model structures on
$S$-modules and $*$-modules, respectively.

Our second main result, Theorem~\ref{thm_C_strict}, lifts the above global model structure on
$*$-modules to a model structure on the category of monoids in $*$-modules. It is Quillen
equivalent to Schwede's global model structure on monoids in orthogonal spaces.
In other words,
monoids in orthogonal spaces and in $*$-modules form equivalent models for the global homotopy
theory of $A_\infty$-spaces.

The global model structure on orthogonal spaces also lifts to commutative monoids. It remains to
be seen whether the analogous result is true for $*$-modules, i.e., whether $*$-modules model the
global homotopy theory of $E_\infty$ spaces.

\textbf{Relation to other work:}
Orthogonal spaces, $\L$-spaces and $*$-modules are the unstable counterparts of
the category of orthogonal spectra, the category of $\L$-spectra, and the category of
$S$-modules, respectively. We refer to \cite{lind:diagram_spaces} for a discussion of
non-equivariant model structures, the relationship with the classical unstable and stable
homotopy categories and further references.
Our main source for properties of $\L$-spaces and $*$-modules is the discussion by Blumberg,
Cohen and Schlichtkrull in Section~4 of \cite{BCS:THH}.

For a fixed group $G$, orthogonal spectra and $S$-modules with
additional structure encoding the $G$-action have been studied equivariantly, see
e.g. \cite{mandellmay:eqvar_orth_spectra, HHR, schwede:lecture_ESHT}.
These additional data are not necessary in global homotopy theory. For
each compact Lie group $G$, the $G$-equivariant homotopy groups of an ordinary
orthogonal spectrum can be defined by evaluating only at $G$-representations. This idea
gives rise to the global homotopy theory of orthogonal spectra and orthogonal
spaces developed by Schwede in his monograph \cite{schwede:GHTv1}.
Schwede's work includes variants that don't take into account all
compact Lie groups, but only a certain family of groups. Hausmann
\cite{hausmann:sym_spectra_global_fin} gave an equivalent description in the
case of all finite groups, using symmetric spectra as a model.

\textbf{Organization:}
Section~\ref{sect:prelim} provides background material on orthogonal spaces, $\L$-spaces and
relevant functors. In Section~\ref{sect:MS}, we discuss Schwede's global model structure
for $\L$-spaces and our global model structure on $*$-modules, assuming the statement of the
Transport Theorem (Theorem~\ref{thm_A_strict}). We lift our model structure and Quillen
equivalences to the level of monoids in Section~\ref{sect:monoids}. Finally, the proof of the
Transport Theorem and other technical details are given in Section \ref{sect:proof_thm_A}.

\textbf{Conventions:}
We work over the category $\U$ of compactly generated weak Hausdorff spaces.
A \emph{model category} is a Quillen model category as defined in
\cite[Def.~3.3]{DS:modelcats}. The definition does not require functorial
factorizations. A \emph{monoidal model category} satisfies the pushout product
axiom and the unit axiom, see \cite[Def.~4.2.6]{hovey:modelcats}. An
$h$-\emph{cofibration} in a model category tensored over $\U$ is a map that
satisfies the homotopy extension property. In diagrams, the upper or left arrow
of an adjunction is the left adjoint.

\textbf{Acknowledgements:}
The present paper arose from the author's Master's thesis \cite{boehme:masters}
supervised by Stefan Schwede at the University of Bonn. The author
would like to thank him for numerous helpful discussions and for suggesting many
of the ideas that led to the results presented here. Moreover, the author is
grateful for many helpful comments and suggestions provided by Jesper Grodal, Markus Hausmann, Manuel
Krannich, Malte Leip, Irakli Patchkoria and an anonymous referee. \\
This research was partly supported by the Danish National Research Foundation
through the Centre for Symmetry and Deformation (DNRF92).

\section{Preliminaries}
\label{sect:prelim}

This section provides some background on the categories of $\L$-spaces,
$*$-modules and orthogonal spaces, as well as on $G$-universes and
universal subgroups. It does not contain any original results. The main sources
are Schwede's preprint \cite{schwede:OOUv1}, his monograph
\cite{schwede:GHTv1} and the article \cite{BCS:THH} by Blumberg, Cohen and
Schlichtkrull.

\subsection{$\L$-spaces and $*$-modules} \label{subsect:LU_Mstar}

Let $\LIE(V, W)$ be the space of linear isometric embeddings $V \to W$ between
two real inner product spaces of finite or countable dimension, topologized as
a subspace of $\U(V, W)$.
Write $\R^\infty := \bigoplus_\N \R$ for the standard inner product space of
countable dimension. The \emph{operad of linear isometric embeddings} $\L$ is given
by spaces $\L(n) = \LIE((\R^\infty)^n, \R^\infty)$, and structure maps induced by direct
sum and composition of maps. It is a (symmetric) $E_\infty$-operad with
$\Sigma_n$-actions via permutation of the $n$ summands of $(\R^\infty)^n$.
The space of unary operations $\L(1)$ is a topological monoid under
composition, and we will study $\L(1)$-equivariant homotopy theory.

\begin{definition}
An \emph{$\L$-space} is a space $X \in \U$ together with a continuous action
from the monoid $\L(1)$. We write $\LU$ for the category of $\L$-spaces and
$\L(1)$-equivariant maps.
\end{definition}

The category $\LU$ is bicomplete where (co-)limits are taken in the category
$\U$ of spaces and equipped with the respective (co-)limit action, because the
forgetful functor to spaces has both adjoints. Moreover, $\LU$ is enriched,
tensored and co-tensored over $\U$.

The \emph{box product} of $\L$-spaces $X$ and $Y$ is the balanced product
\[ X \Lbox Y := \L(2) \times_{\L(1)^2} (X \times Y) \]
with respect to the right $\L(1)^2$-action on $\L(2)$ given by precomposition on either summand of $\R^\infty \oplus \R^\infty$. The
space $X \Lbox Y$ is an $\L$-space via the left $\L(1)$-action on $\L(2)$ given by postcomposition.

\begin{lemma}[Hopkins' lemma, see \cite{EKMM},~Lemma~I.5.4] \label{Hopkins}
For $m, n \geq 1$, the space $\L(m+n)$ is a split coequalizer of the diagram
\begin{center}
$\xymatrix@M=10pt{
\L(2) \times \L(1)^2 \times (\L(m) \times \L(n))
\ar@<0.7ex>[r] \ar@<-0.7ex>[r] &
\L(2) \times (\L(m) \times \L(n)),
}$
\end{center}
hence $\L(m) \Lbox \L(n) \cong \L(m+n)$ as
$\L$-spaces.
\end{lemma}

The box product admits coherent associativity and commutativity isomorphisms
and a right adjoint $\Fbox{Y}{-}$ for the functor $- \Lbox Y \colon \LU \to \LU$, see
\cite[Sect.~4.1]{BCS:THH} and \cite[Def.~2.19]{boehme:masters}; cf.~also
\cite[Sect.~I.5]{EKMM}. We will give an explicit description of $\Fbox{Y}{-}$ in Lemma~\ref{F
explicitly} and record some of its properties in Proposition~\ref{prop properties Fbox}.
There is a natural transformation
\[ \lambda_{X,Y} \colon \hspace{1pc} X \Lbox Y = \L(2) \times_{\L(1)^2} (X
\times Y) \hspace{1pc} \to \hspace{1pc} X \times Y \hspace{3pc} \]
\[ \hspace{8.5pc} \left[ \psi_1 \oplus \psi_2, (x, \, y) \right] \hspace{1.5pc}
\mapsto \hspace{0.5pc} (\psi_1 \cdot x, \, \psi_2 \cdot y). \]
which restricts to a \emph{unital transformation}
\[ \lambda_{X} \colon \hspace{1pc} X \Lbox * = \L(2) \times_{\L(1)^2} (X \times
*) \hspace{1pc} \to \hspace{1pc} X \hspace{3pc} \]
\[ \hspace{7pc} \left[ \psi_1 \oplus \psi_2, (x, \, *) \right] \hspace{1pc}
\mapsto \hspace{1pc} \psi_1 \cdot x. \]
Here we used that each linear isometric embedding $\psi \colon \R^\infty
\bigoplus \R^\infty \to \R^\infty$ is given as $\psi_1 \oplus \psi_2$, where the
$\psi_i \in \L(1)$ have orthogonal images.
Unfortunately, $\lambda$ fails to be an isomorphism for all $\L$-spaces: For
instance, all linear maps in the image of $\lambda_{\L(1)}$ have an
infinite-dimensional orthogonal complement, hence it is not surjective.

However, $\lambda_X$ is always a weak equivalence of underlying spaces, see
\cite[Sect.~4.1]{BCS:THH}, and it satisfies an even stronger, equivariant
notion of equivalence, see Proposition \ref{prop:lambda}. In order to be able
to refer to this situation, we make the following definition.

\begin{definition} \label{def:weakSMC}
A relative category $(\mathcal{C}, \mathcal{W})$ is called a \emph{weak (closed)
symmetric monoidal category} if it is (closed) symmetric monoidal in the usual
sense except that the left and right unital transformations are only required
to lie in the class of weak equivalences $\mathcal{W}$, not necessarily in the
class of isomorphisms.
\end{definition}

\begin{rem}
Note that the usual definition of a monoid in a symmetric monoidal category in terms of two
commutative diagrams still makes sense in a weak symmetric monoidal category. By a slight abuse of
terminology, we will simply call such an object a ``monoid'' instead of a ``weak monoid''. Similarly, it makes sense to speak of monoidal functors between weak symmetric monoidal categories, and monoidal model
structures on weak symmetric monoidal categories.
\end{rem}

\begin{definition}
A $*$-\emph{module} is an $\L$-space $M$ such that $\lambda_M$ is an
isomorphism of $\L$-spaces.
\end{definition}

Surprisingly, the quotient $* \Lbox * = \L(2)/\L(1)^2$ is
trivial, as proven in \cite[Lemma~I.8.1]{EKMM}. Consequently, the functor $-
\Lbox *$ on $\L$-spaces takes values in $*$-modules, and the box product
restricts to a well-defined product on $\Mstar$, which we denote by the same
symbol $\Lbox$. So the category $\LU$ is a weak closed symmetric monoidal category, and then
$\Mstar$ is a symmetric monoidal category in the usual sense. The latter is
also closed, as follows formally from Proposition \ref{mirror_image} below.

Dually, we let $\Mcostar$ be the full subcategory of those $\L$-spaces such that
the adjoint $\bar{\lambda}_Y \colon Y \to \Fbox{*}{Y}$ is an isomorphism, and
refer to its objects as \emph{co-unital $\L$-spaces} or \emph{co-$*$-modules}.
The functor $\Fbox{*}{-}$ on $\LU$ takes values in $\Mcostar$.

The following collection of statements from \cite[Sect.~4.3]{BCS:THH} is an
easy exercise in elementary category theory. It is the unstable analogue of a
similar ``mirror image'' argument for $S$-modules,
cf.~\cite[Sect.~II.2]{EKMM}.

\begin{prop} \label{mirror_image}
The categories $\Mstar$ and $\Mcostar$ of unital and co-unital $\L$-spaces,
respectively, are ``mirror image subcategories'' in the following sense:
\begin{enumerate}[a)]
    \item All pairs of functors in the diagram below form adjunctions (where
    upper arrows and arrows on the left hand side are left adjoints).
    \begin{center}
	$\xymatrix@M=12pt@C=10pc@R=6pc{
	\LU \ar@<0.7ex>[r]^{- \Lbox *} \ar@{->}@<0.7ex>[dr]^{- \Lbox *}
	\ar@<-0.7ex>[d]_{\Fbox{*}{-}} &
	\Mstar \ar@<0.7ex>[l]^{\Fbox{*}{-}} \ar@{_{(}->}@<-0.7ex>[d] \\
	\Mcostar \ar@{_{(}->}@<-0.7ex>[u] \ar@<0.7ex>[r]^{- \Lbox *} &
	\LU \ar@<-0.7ex>[u]_{- \Lbox *} \ar@{->}@<0.7ex>[ul]^{\Fbox{*}{-}}
	\ar@<0.7ex>[l]^{\Fbox{*}{-}}
	}$
	\end{center}
	\item The subdiagrams of left-adjoint (respectively right-adjoint) functors
    commute up to natural equivalence.
    \item The categories $\Mstar$ and $\Mcostar$ are bicomplete. Colimits in
    $\Mstar$ are created in $\LU$, limits are obtained by applying $- \Lbox *$
    to limits in $\LU$; dually for $\Mcostar$.
    \item The diagonal adjunction (co-)restricts to an equivalence of categories
    \begin{center}
	$\xymatrix@M=10pt@!=6pc{
	\Mcostar \ar@<0.7ex>[r]^{- \Lbox *} &
	\Mstar. \ar@<0.7ex>[l]^{\Fbox{*}{-}}
	}$
	\end{center}
\end{enumerate}
\end{prop}

\subsection{Completely universal subgroups}
In this short section, we recall that every compact Lie group is isomorphic to
an actual subgroup of $\L(1)$, a so-called \emph{completely universal subgroup}. Thus, all
compact Lie groups act simultaneously on each $\L$-space $X$; these actions are
compatible in the sense that they all are restrictions of the same action of
$\L(1)$ on $X$.

\begin{definition}
Let $U_G$ be an orthogonal $G$-representation of countable dimension. We say
that $U_G$ is 
\begin{enumerate}[i)]
  \item a \emph{$G$-universe} if it contains a 1-dimensional trivial subrepresentation and has the property that for each
  finite-dimensional $G$-representation $V$ that embeds into $U_G$, the representation $\bigoplus_\N V$ also embeds into $U_G$,
  \item a \emph{complete $G$-universe} if it is a $G$-universe that contains one
  copy, and hence countably many copies of each irreducible $G$-representation.
\end{enumerate}
\end{definition}

\begin{definition}[\cite{schwede:OOUv1}, Def.~1.4]
A compact subgroup $G \leq \L(1)$ is called \emph{completely universal} if it admits the structure of a
compact Lie group (necessarily unique, see \cite[Prop.~3.12]{broeckertomdieck})
such that under the tautological action, $\R^\infty$ becomes a complete $G$-universe.
\end{definition}

\begin{rem}
The completely universal subgroups are just called ``universal subgroups'' in \cite{schwede:OOUv1}, but we will use the more precise
terminology.
\end{rem}


\begin{lemma}[cf.~\cite{schwede:OOUv1}, Prop.~1.5]
The equivalence classes of completely universal subgroups of $\L(1)$ under
conjugation by invertible elements of $\L(1)$ biject with the isomorphism
classes of compact Lie groups.
\end{lemma}


In Section \ref{sect:MS}, we will introduce various notions of equivalences
and fibrations detected on $G$-fixed points for all completely universal
subgroups $G \leq \L(1)$.

\subsection{Orthogonal spaces} \label{subsect:IU}
Write $\I$ for the category of finite-dimensional real inner product spaces with
morphisms the linear isometric embeddings. It is enriched over spaces, see
the beginning of Subsection \ref{subsect:LU_Mstar}.

\begin{definition}
An \emph{orthogonal space} is a continuous functor $Y \colon \I \to \U$. We
write $\IU$ for the category of orthogonal spaces and natural transformations.
\end{definition}

The category $\IU$ is bicomplete, with (co-)limits taken objectwise.
Moreover, it is tensored and co-tensored over $\U$ where, for $Y \in \IU, \, A
\in \U$, the tensor orthogonal space $Y \times A$ sends $V \in I$ to $(Y \times A)(V) :=
Y(V) \times A$. Equivalently, we can regard $A$ as the constant orthogonal space
with value $A$ and form the product in $\IU$.

The category of orthogonal spaces is a closed symmetric
monoidal category under the \emph{box product}, which is the Day convolution
product with respect to direct sum of vector spaces in $\I$ and the product in
$\U$, see \cite[Sect.~1.3, App.~C]{schwede:GHTv1} for further details.
A unit is given by the constant one-point orthogonal space $\one$.

Following Schwede, we take \emph{global equivalences} of orthogonal spaces to be
those morphisms that, for each compact Lie group $G$, induce $G$-weak equivalences on
homotopy colimits along $G$-representations. The precise definition is given
in more elementary terms, cf.~\cite[Rem.~1.1.4]{schwede:GHTv1}.

\begin{definition}[\cite{schwede:OOUv1}, Def.~3.4]
A morphism $f \colon X \to Y$ of orthogonal spaces is a \emph{global
equivalence} if for any compact Lie group $G$, any orthogonal
$G$-representation $V$ of finite dimension, any $k \geq 0$ and any
commuting square
\begin{center}
$ \xymatrix@M=8pt{
	S^{k-1} \ar[r]^(.45)\alpha \ar[d]_{\mathrm{incl}} &
	X(V)^G \ar[d]^{f(V)^G} \\
	D^k \ar[r]_(.45)\beta &
	Y(V)^G
}$
\end{center}
there is a \fd $G$-representation $W$, a $G$-equivariant linear isometric
embedding $\varphi \colon V \to W$ and a map $\lambda \colon D^k \to X(W)^G$
such that in the extended diagram
\begin{center}
$ \xymatrix@M=8pt@C=3pc{
	S^{k-1} \ar[r]^(.45)\alpha \ar[d]_{\mathrm{incl}} &
	X(V)^G \ar[r]^{X(\varphi)^G} &
	X(W)^G \ar[d]^{f(W)^G} \\
	D^k \ar[r]_(.45)\beta \ar@{-->}[urr]^(.45)\lambda &
	Y(V)^G \ar[r]_{Y(\varphi)^G} &
	Y(W)^G
}$ \end{center}
the upper triangle commutes strictly and the lower triangle commutes up to
homotopy relative to $S^{k-1}$.
\end{definition}

\begin{thm}[Global model structures for orthogonal spaces,
cf.~\cite{schwede:GHTv1}, Thm.~1.2.21, Prop.~1.2.23]
\label{thm_MS_IU}
The global equivalences are part of two proper, topological, cofibrantly
generated model structures on the category of orthogonal spaces, the
\emph{(absolute) global model structure} $(\IU)_\mathrm{abs}$ and the
\emph{positive global model structure} $(\IU)_\mathrm{pos}$.
\end{thm}

We omit the description of the (co-)fibrations in the two global model structures
since we will not make use of them explicitly. As usual, the positive variant
has a better behaviour \wrt commutative monoids, and for different reasons, it
is also necessary for us to work with the positive global model structure
throughout the paper, see Remark \ref{rem_positive}. The absolute model
structure will only appear in Section \ref{sect:monoids}.

Note that both global model structures are monoidal \wrt the box product of orthogonal spaces: The
pushout product axiom is proven in \cite[Prop.~1.4.12 iii), iv)]{schwede:GHTv1}, while the unit axiom follows from
\cite[Thm.~1.3.2 ii)]{schwede:GHTv1}.

The categories $\IU$ and $\LU$ can be connected by an adjoint pair of
functors. By general theory, the right exact enriched functors (i.e., those
which preserve colimits and tensors) $\mathcal{DU} \to \mathcal{C}$ out of
a category of diagram spaces into a category $\mathcal{C}$ that is enriched and cocomplete in the enriched sense agree with
the enriched functors $\mathcal{D}^{op} \to \mathcal{C}$ up to isomorphism of
categories; see \cite[Sect.~I.2]{mandellmay:eqvar_orth_spectra} and note that the results also
apply in the unbased case. \\
Lind \cite[Def.~8.2]{lind:diagram_spaces} defines a functor $\qu^*
\colon \mathcal{I}^{op} \to \LU$ that sends $V$ to $\L(V \otimes \R^\infty,
\R^\infty)$; it is strong symmetric monoidal by Lemma \ref{Hopkins}.
The results of \cite{mandellmay:eqvar_orth_spectra} then yield an adjunction
\begin{center}
	$\xymatrix@M=10pt@!=6pc{
	\IU \ar@<0.7ex>[r]^{\qu} &
	\LU. \ar@<0.7ex>[l]^{\qu^\#}
	}$
\end{center}
where the left adjoint $\qu$ is given as an enriched coend $\qu^* \otimes_\I
(-)$ and the right adjoint is $\qu^\# X(V) = \LU(\qu^*(V), X)$. The first is
strong, the latter lax symmetric monoidal.

The functor $\qu^* \colon \I^{op} \to \LU$ can be replaced by
$\qu^*_* \colon V \mapsto \L(V \otimes \R^\infty, \R^\infty) \Lbox *$. Then $\qu^*_*$ takes
values in $*$-modules and yields an adjunction
\begin{center}
	$\xymatrix@M=10pt@!=6pc{
	\IU \ar@<0.7ex>[r]^{\qu_*} &
	\Mstar. \ar@<0.7ex>[l]^{\qu_*^\#}
	}$
\end{center} 
defined in the same way as before. Again, the left adjoint is strong, the right
adjoint lax symmetric monoidal. By \cite[Lemma~8.6]{lind:diagram_spaces}, this
pair of functors agrees, up to natural equivalence, with the
composition of adjunctions
\begin{center}
$\xymatrix@M=10pt@C=4pc{
\IU \ar@<0.7ex>[r]^(.5){\qu} & \LU
\ar@<0.7ex>[l]^(.5){\qu^\#} \ar@<0.7ex>[r]^(.5){- \Lbox *} &
\Mstar \ar@<0.7ex>[l]^(.48){\Fbox{*}{-}}.
}$ 
\end{center}

\begin{rem} \label{functor_O}
There is another interesting choice of a functor $\IU \to \LU$.
For an orthogonal space $Y$, the colimit $Y(\R^\infty) := \colim_V Y(V)$ taken
over all finite-dimensional inner product subspaces $V \subseteq \R^\infty$
(or equivalently, all standard Euclidean spaces $\R^n$) has a canonical $\L$-space
structure, see \cite[Constr.~3.2]{schwede:OOUv1}.
The resulting functor $\obb \colon \IU \to \LU$ is induced by $\obb^* \colon \I^{op} \to \LU$
sending $V \in \I$ to $\LIE(V, \R^\infty)$, see \cite[Lemma~9.6]{lind:diagram_spaces}.
In unpublished work, Schwede proved that $\obb$ is strong symmetric monoidal.
It follows formally that its right
adjoint is a lax symmetric monoidal functor. \\
Any choice of a one-dimensional subspace of $\R^\infty$ defines a linear
isometric embedding $V \to V \otimes \R^\infty$, hence a natural transformation
$\xi^* \colon \qu^*(V) = \L(V \otimes \R^\infty, \R^\infty) \to \L(V, \R^\infty)
\cong \obb^*(V)$ which in turn determines a natural map $\xi = \xi^*
\otimes_I (-) \colon \qu \to \obb$. The latter is symmetric monoidal; moreover,
it is a global equivalence on cofibrant objects in the absolute model structure on
orthogonal spaces, see 
\cite[Prop.~3.7]{schwede:OOUv1}.
A precursor of the last statement was \cite[Lemma~9.7]{lind:diagram_spaces}.
\end{rem}

\section{Global model structures for $\LU$ and $\Mstar$}
\label{sect:MS}

We recall Schwede's model structures for $\L$-spaces from
\cite{schwede:OOUv1} and derive our first main result, Theorem~\ref{thm_B_strict}. It establishes
the global model structure for $*$-modules which is Quillen equivalent to orthogonal spaces
via the functor $\qu_*$.

The following notions of equivalences and fibrations of $\L$-spaces will also be
used for maps of $*$-modules by viewing them as maps in $\LU$.

\begin{definition}[\cite{schwede:OOUv1}, Def.~1.6, Def.~1.8] \label{def:equivalences}
Let $\CL$ denote the set of all compact Lie subgroups of $\L(1)$.
A map $f \colon X \to Y$ of $\L$-spaces is called
\begin{itemize}
  \item a $\CL$-\emph{equivalence} (respectively $\CL$-\emph{fibration}) if the map $f^G \colon X^G \to Y^G$ is a weak homotopy
  equivalence (respectively Serre fibration) of spaces for all compact Lie subgroups $G \leq \L(1)$;
  \item a \emph{global equivalence} if $f^G \colon X^G \to Y^G$ is a weak
  homotopy equivalence of spaces for all completely universal subgroups $G \leq
  \L(1)$;
  \item a \emph{strong global equivalence} if the map $f$, considered as a map of $G$-spaces, is a $G$-equivariant
  homotopy equivalence for all completely universal subgroups $G \leq \L(1)$.
\end{itemize}
\end{definition}

\begin{prop} \label{prop:lambda}
The natural map of $\L$-spaces $\lambda_{X,Y} \colon X \Lbox Y \to X \times Y$
is a strong global equivalence for all $X, Y \in \LU$. Consequently, so is the
adjoint map $\bar{\lambda}_X \colon X \to \Fbox{*}{X}$. Both functors $(-)
\Lbox *$ and $\Fbox{*}{-}$ preserve and reflect (strong) global equivalences. For
all $Z \in \LU$, the functor $(-) \Lbox Z$ preserves (strong) global
equivalences.
\end{prop}

\begin{proof}
The first part is \cite[Thm.~1.21]{schwede:OOUv1}. In combination with the
2-out-of-3 property and the following diagram, it implies the second statement;
the third then follows immediately.
\begin{center}
$ \xymatrix@M=10pt@C=4pc{
	X \Lbox * \ar[r]^{\lambda_X}_{\sim} \ar[d]^{\cong}_{\bar{\lambda}_X \Lbox *} &
	X \ar[d]^{\bar{\lambda}_X} \\
	\Fbox{*}{X} \Lbox * \ar[r]^(.55){\sim}_(.55){\lambda_{\Fbox{*}{X}}} &
	\Fbox{*}{X}
}$
\end{center} 
Now let $Z \in \LU$ be arbitrary. If $f$ is a (strong) global equivalence, then
so is $f \times Z$, hence also $f \Lbox Z$.
\end{proof}


Recall that for any $G \leq \L(1)$, the $\L$-space $\L(1)/G$ represents the
fixed point functor $(-)^G \colon \LU \to \U$.
The collection of fixed point functors associated to $G \in \CL$ gives rise to the following model structure on $\L$-spaces.  

\begin{prop}[$\CL$-projective model structure for $\L$-spaces, 
\cite{schwede:OOUv1}, Prop.~1.11] \label{prop_univMS_LU}
There is a proper topological model structure $(\LU)_{\CL}$ on the
category of $\L$-spaces with weak equivalences and fibrations the $\CL$-equivalences
and $\CL$-fibrations. It is cofibrantly generated with sets of
generating (acyclic) cofibrations
obtained by tensoring the standard generating (acyclic) cofibrations
for spaces $S^{n-1} \to D^n$ (respectively $D^n \times 0 \to D^n \times I$) with
$\L$-spaces of the form $\L(1)/G$,
where $G$ runs through all compact Lie subgroups of $\L(1)$.
\end{prop}

The $\CL$-projective model structure seems unlikely to be Quillen equivalent to $\IU$
with its positive global model structure, but one can perform a left Bousfield localization such that
the weak equivalences become precisely the class of global equivalences. This
detour is necessary in order to guarantee that the adjunction to orthogonal
spaces becomes a Quillen adjunction. We refer the reader to \cite[Section~1]{schwede:OOUv1} for a
detailed discussion of this Bousfield localization and an explicit description of the local objects.

%
%

\begin{thm}[Global model structure for $\L$-spaces, see
\cite{schwede:OOUv1}, Thm.~1.20] \label{thm_globalMS_LU}
There is a cofibrantly generated proper topological
model structure $(\LU)_\mathrm{gl}$ on the category of $\L$-spaces with weak equivalences the
global equivalences and cofibrations as in $(\LU)_{\CL}$. Every $\CL$-cofibration is an $h$-cofibration of $\L$-spaces and a closed
embedding of underlying spaces.
\end{thm}

\begin{prop}
The model structure $(\LU)_\mathrm{gl}$ is a monoidal model category.
\end{prop}

\begin{proof}
The pushout product axiom is proven in \cite[Prop.~1.22]{schwede:OOUv1}, the
unit axiom follows from Proposition \ref{prop:lambda}.
\end{proof}

%
%


The global model structures for orthogonal spaces and $\L$-spaces model the same homotopy theory.

\begin{thm}[\cite{schwede:OOUv1}, Thm.~3.9] \label{thm Quillen equiv IU LU}
The adjunction
\begin{center}
$ \xymatrix@M=10pt@R=3pc@C=4pc{
	(\IU)_\mathrm{pos} \ar@<0.7ex>[r]^{\qu} &
	(\LU)_\mathrm{gl} \ar@<0.7ex>[l]
}$
\end{center}
is a Quillen equivalence.
\end{thm}


\begin{rem} \label{rem_positive}
The functor $\qu^\#$ is not a right Quillen functor anymore if we use the
absolute model structure on orthogonal spaces instead:
If $\qu^\#X$ is fibrant in the absolute model
structure, then \cite[Def.~1.2.12]{schwede:GHTv1} (for $G$ the trivial group and $V \to W$
the inclusion $0 \to \R$) implies that the inclusion of fixed points $X^{\L(1)} \to X$ must be a
weak homotopy equivalence of spaces. It seems very unlikely that this could be true for all
fibrant $\L$-spaces $X$.
\end{rem}

Assuming the Transport Theorem (Theorem \ref{thm_A_strict}), we will now prove the
following:

\begin{thm}[Global model structure for $*$-modules] \label{thm_B_strict}
There is a cofibrantly generated proper topological model structure on the
category $\Mstar$ of $*$-modules, the \emph{global model structure}
$(\Mstar)_\mathrm{gl}$. Its weak equivalences are the global equivalences of
underlying $\L$-spaces, its fibrations are detected by the functor $\Fbox{*}{-}
\colon \Mstar \to (\LU)_\mathrm{gl}$. Let $I$ and $J$ be any sets of generating
(acyclic) cofibrations for $(\LU)_\mathrm{gl}$, then $I \Lbox *$ and $J \Lbox
*$ are generating (acyclic) cofibrations for $(\Mstar)_\mathrm{gl}$. \\
Moreover, the global model structure for $\Mstar$ is monoidal and satisfies the
monoid axiom \cite[Def.~3.3]{schwedeshipley:algsmods}
\wrt $\Lbox$. It fits into the following commutative (up to natural isomorphism)
triangle of monoidal Quillen equivalences:
\begin{center}
$ \xymatrix@M=10pt@R=3pc@C=4pc{
	(\IU)_\mathrm{pos} \ar@<0.7ex>[r]^{\qu}
	\ar@<0.7ex>[dr]^(.55){\qu_*} &
	(\LU)_\mathrm{gl} \ar@<0.7ex>[l] \ar@<-0.7ex>[d] \\
	\, &
	(\Mstar)_\mathrm{gl} \ar@<-0.7ex>[u]_{\Fbox{*}{-}}
	\ar@<0.7ex>[ul]
}$
\end{center}
\end{thm}

\begin{proof}
The global model structure obviously satisfies the
requirements of the Transport Theorem (Theorem \ref{thm_A_strict}), which immediately
implies the existence and properties of the model structure
$(\Mstar)_\mathrm{gl}$.
It also proves that the vertical adjunction is a
Quillen equivalence. The horizontal adjunction is a Quillen equivalence by Theorem~\ref{thm
Quillen equiv IU LU}, and we have already seen that all adjunctions are monoidal.
\end{proof}

\begin{rem}
There is a variant of Theorem~\ref{thm_B_strict} \wrt the functor $\obb \colon \IU \to \LU$
introduced in Remark \ref{functor_O}: It is possible to establish a model
structure on $\L$-spaces with weak equivalences the global equivalences and such
that $\obb$ is a left Quillen equivalence \wrt the absolute global model
structure on orthogonal spaces. This model structure also satisfies the
hypotheses of the Transport Theorem, but is harder to work with as the cofibrations
cannot only be characterized in terms of fixed points of group
actions. It also lifts to monoids and the analogue of Theorem~\ref{thm_C_strict} holds. Details
can be found in the author's (unpublished) Master's thesis \cite{boehme:masters}.
\end{rem}

\begin{rem}
The diagram in Theorem \ref{thm_B_strict} can be extended to the
right: By a version of Elmendorf's theorem, $(\LU)_\mathrm{gl}$ is Quillen
equivalent to a model category of ``systems of global fixed point sets''. As
usual, these are diagram spaces indexed on the opposite of a suitable
``global'' orbit category. We refer to \cite[Section~2]{schwede:OOUv1} for details.
\end{rem}


\section{Monoids and modules in global homotopy theory}
\label{sect:monoids}

Monoids \wrt $\Lbox$ and their modules have been described
non-equivariantly by Blumberg, Cohen and Schlichtkrull, see
\cite[Thm.~4.18]{BCS:THH}. We describe ``global'' analogues of their result and
prove our second main result, Theorem~\ref{thm_C_strict}.

Recall from Section \ref{sect:prelim} that $\L$ denotes the operad of linear
isometric embeddings of $\R^\infty$. The following identifications are a
consequence of Hopkins' Lemma \ref{Hopkins}.

\begin{prop}[\cite{BCS:THH},~Prop.~4.7] \label{A_infty}
The category of $A_\infty$-spaces structured by $\L$ (considered as a
non-symmetric operad) is isomorphic to the category of $\Lbox$-monoids in $\LU$.
The category of $E_\infty$-spaces structured by $\L$ (considered as a
symmetric operad) is isomorphic to the category of commutative $\Lbox$-monoids
in $\LU$.
\end{prop}

\begin{cor}[\cite{BCS:THH},~Sect.~4.4]
The $\Lbox$-monoids in $\Mstar$ are those $A_\infty$-spaces which are
$*$-modules. The functor $- \Lbox * \colon \LU \to \Mstar$ takes $\Lbox$-monoids
in $\LU$ to $\Lbox$-monoids in $\Mstar$ and the natural map $\lambda_X \colon X
\Lbox * \to X$ is a map of $\Lbox$-monoids if $X$ is a $\Lbox$-monoid. The
analogous statement is true for commutative monoids and $E_\infty$-spaces.
\end{cor}

In \cite{schwedeshipley:algsmods}, Schwede and Shipley describe sufficient
conditions for a cofibrantly generated monoidal model structure
to lift to the associated categories of $R$-modules and
$R$-algebras, respectively, where $R$ is any (commutative) monoid. When applied
to the global model structure on $*$-modules, this yields:

\begin{thm} \label{MS_mon_mod_Mstar} 
Consider the category of $*$-modules equipped with the global model structure
and let $R$ be a $\Lbox$-monoid in $\Mstar$. Call a morphism of $R$-algebras a
weak equivalence (respectively fibration) if it is a weak equivalence
(respectively fibration) of underlying $*$-modules. With respect to these
classes of morphisms, the following holds:
\begin{enumerate}[1)]
    \item The category of left $R$-modules is a
    cofibrantly generated model category.
    \item If $R$ is commutative, then the category of
    $R$-modules is a cofibrantly generated model category satisfying the pushout product axiom and
    the monoid axiom.
    \item If $R$ is commutative, then the category of
    $R$-algebras is a cofibrantly generated model category. If the source of a
    cofibration of $R$-algebras is cofibrant as an $R$-module, then the map is
    a cofibration of $R$-modules.
\end{enumerate}
In all cases, sets of generating cofibrations and
acyclic cofibrations are given by the images of generating sets for $\Mstar$ under the
free functor.
\end{thm}

For $R = *$, the category of $R$-algebras is the category of
$\Lbox$-monoids. It has a cofibrantly generated model structure by part 3) of
the theorem.

\begin{proof}
We check the hypotheses of \cite[Thm.~4.1]{schwedeshipley:algsmods}. As
explained in \cite[Rem.~4.2]{schwedeshipley:algsmods}, the smallness assumption
can be weakened; it then follows from the fact that the forgetful functors from
$R$-modules and monoids, respectively, commute with filtered colimits, and from
Lemma \ref{smallness}. \\
By part h) of Theorem \ref{thm_A_strict}, $(\Mstar)_\mathrm{gl}$
satisfies the monoid axiom as defined in \cite[Def.~3.3]{schwedeshipley:algsmods}.
\end{proof}

\begin{thm} \label{MS_mon_mod_LU}
The analogue of Theorem \ref{MS_mon_mod_Mstar} \wrt the monoidal model
category $(\LU)_\mathrm{gl}$ is true.
\end{thm}

\begin{proof}
A close inspection of the proof of
\cite[Thm.~4.1]{schwedeshipley:algsmods} shows that the first two
statements do not require that the unital transformation is an isomorphism, so
these hold because $(\LU)_\mathrm{gl}$ satisfies the monoid axiom, see part h)
of Theorem \ref{thm_A_strict}. The proof of the third statement makes use of the
unital isomorphism in order to verify that all relative $J_T$-cell complexes
are weak equivalences. We will give an alternative proof of this fact
instead: \\
Here, $T \colon X \mapsto \coprod_{n \geq 0} X^{\Lbox n}$ is the composition of
the free monoid functor with the forgetful functor, $J$ is any set of
generating acyclic cofibrations for $(\LU)_\mathrm{gl}$, and $J_T$ denotes its
image under
$T$. All maps in $J$ are $h$-cofibrations (i.e., have the homotopy extension
property) and global equivalences. For each $Z \in \LU$, the left adjoint
functor $Z \Lbox (-) \colon \LU \to \LU$ preserves these properties by
Proposition \ref{prop:lambda} and Lemma
\ref{preservation_h-cofibrations}. Thus, for a map $j \colon A \to B$ in $J$
and $n \geq 2$, we can write the $n$-th summand $j^{\Lbox n}$ of $T(j)$ as a
composition
\[ (j \Lbox A^{\Lbox (n-1)}) \circ (B \Lbox j \Lbox A^{\Lbox (n-2)})
\circ \ldots \circ (B^{\Lbox (n-1)} \Lbox j). \]
of maps which are both $h$-cofibrations and global equivalences. These
properties are stable under composition and coproducts, hence $T(j)$ has both
properties.
Moreover, the class of $h$-cofibrations which are
global equivalences is closed under cobase change and transfinite composition,
thus each morphism in $J_T-cell$ is a global equivalence. Smallness is not an
issue because all $\L$-spaces are small relative to closed embeddings
(Lemma \ref{smallness}), and so relative to all images of cofibrations under
$T$.
\end{proof}

\begin{thm} 
\label{MS_mon_mod_IU}
The analogue of Theorem \ref{MS_mon_mod_Mstar} \wrt the monoidal model
categories $(\IU)_\mathrm{abs}$ and $(\IU)_\mathrm{pos}$ is true.
\end{thm}

\begin{proof}
Every acyclic cofibration in the positive global model structure on $\IU$ is an
acyclic cofibration in the absolute global model structure. The latter
satisfies the monoid axiom, see \cite[Prop.~1.4.13]{schwede:GHTv1}, hence
so does the former. Again, \cite[Thm.~4.1]{schwedeshipley:algsmods} applies.
\end{proof}

We can now state our second main result.

\begin{thm}[Global model structure for monoids in $*$-modules] \label{thm_C_strict}
The triangle of monoidal Quillen equivalences from Theorem \ref{thm_B_strict}
gives rise to a triangle of Quillen equivalences
between the respective model structures on categories of monoids.
\end{thm}

\begin{proof}
For all three categories, the forgetful functors from monoids preserve and reflect fibrations and weak
equivalences. Thus for all three adjunctions from Theorem~\ref{thm_B_strict}, the lifted right adjoints are
always right Quillen functors, and it remains to show that they are Quillen equivalences. \\
The induced adjunction between monoids in $\LU$ and $\Mstar$ is a Quillen equivalence because the
functor $- \Lbox *$ preserves and reflects global equivalences and the
 counit $\Fbox{*}{X} \Lbox * \to X$ is an isomorphism for all $X \in
\Mstar$, see Proposition \ref{mirror_image}. \\
Now consider the Quillen adjunction between monoids in $\IU$ and $\LU$.
We will show that the derived adjunction is an equivalence of categories. More precisely, we will mimick parts of the proof of
\cite[Thm.~3.9]{schwede:OOUv1} and show that
\begin{enumerate}[(I)]
  \item the derived right adjoint reflects isomorphisms, and
  \item the unit of the derived adjunction is a natural isomorphism.
\end{enumerate}
Since we are working with model structures on monoids created by the forgetful functor, statement (I) immediately follows from fact
(a) in the proof of \cite[Thm.~3.9]{schwede:OOUv1}. In order to prove statement (II), it suffices to show that for all positively
cofibrant monoids $M$ in $\IU$ and some (hence any) fibrant replacement $(-)_{f, mon}$ in the category of monoids in $\LU$, the
underlying map of orthogonal spaces $M \to \qu^\#((\qu(M))_{f, mon})$ is a global equivalence. The monoidal unit in $\IU$ is
absolutely cofibrant, hence the underlying orthogonal space of any positively cofibrant monoid $M$ is absolutely cofibrant, see
Theorem~\ref{MS_mon_mod_IU} and part 3) of Theorem~\ref{MS_mon_mod_Mstar}. Now fact (b) in the proof of
loc.~cit.~asserts that for all positively cofibrant orthogonal spaces $A$ and some (hence any)
fibrant replacement $(-)_{f}$ in $\LU$, the map of orthogonal spaces $A \to \qu^\#((\qu(A))_{f})$ is a global equivalence. Moreover, the
proof given in loc.~cit.~works without changes for absolutely cofibrant orthogonal spaces $A$. As any fibrant replacement of monoids
in $\LU$ is also a fibrant replacement of the underlying $\L$-spaces, we see that our statement (II) follows.
\\
Finally, the adjunction between monoids in $\IU$ and $\Mstar$ is a Quillen equivalence as the composition of two Quillen
equivalences.
\end{proof}

In light of Proposition \ref{A_infty}, Theorem~\ref{thm_C_strict} states that there is an
unambiguous global homotopy theory of $A_\infty$-spaces. We don't know if
this statement is true for $E_\infty$-spaces: The positive global model
structure $(\IU)_\mathrm{pos}$ lifts to commutative monoids, see
\cite[Thm.~2.1.15]{schwede:GHTv1}, but it remains open whether the same
holds for $(\Mstar)_\mathrm{gl}$. The difficulty is in showing that the functor
$(-)^{\Lbox n}/\Sigma_n$ takes acyclic cofibrations to global equivalences.


\section{The Transport Theorem}
\label{sect:proof_thm_A}

We finally give a precise statement and proof of the Transport Theorem.
Throughout this section, let $F$ denote the functor $\Fbox{*}{-} \colon \LU \to
\Mcostar$ and let $R$ be its right adjoint, the forgetful functor $\Mcostar \to
\LU$.

\begin{thm}[Transport Theorem] \label{thm_A_strict}
Let $(\LU)_a$ be any model structure on the category $\LU$
of $\L$-spaces such that
\begin{enumerate}[i)]
  \item it is cofibrantly generated, with sets of generating cofibrations and
  acyclic cofibrations denoted by $I$ and $J$, respectively
  \item all morphisms in $I$ (and hence in $J$) are closed embeddings of
  underlying spaces
  \item the class $\mathcal{W}$ of weak equivalences contains all
  strong global equivalences (in the sense of Definition
  \ref{def:equivalences})
  \item the class of morphisms which are simultaneously weak equivalences and
  closed embeddings is closed under transfinite composition.
\end{enumerate}
Then the category of $*$-modules $\Mstar$ admits a model
structure $(\Mstar)_a$ satisfying the following properties:
\begin{enumerate}[a)]
  \item It is cofibrantly generated, with sets of generating cofibrations and
  acyclic cofibrations given by $I \Lbox *$ and $J \Lbox *$, respectively.
  \item The weak equivalences are precisely those morphisms of $*$-modules which
  are sent to $\mathcal{W}$ under the forgetful functor to $\LU$.
  \item Fibrations are detected by the functor $\Fbox{*}{-} \colon \Mstar \to
  (\LU)_a$.
  \item The adjunction
	\begin{center}
	$ \xymatrix@M=10pt@C=3pc{
		(\LU)_a \ar@<0.7ex>[r]^(0.5){- \Lbox *} &
		(\Mstar)_a \ar@<0.7ex>[l]^(0.5){\Fbox{*}{-}}
	}$
	\end{center}
	is a Quillen equivalence.
\end{enumerate}
Moreover:
\begin{enumerate}[a)] \setcounter{enumi}{4}
  \item If $(\LU)_a$ is right proper, then so is $(\Mstar)_a$. If $(\LU)_a$ is a
  topological model category, then so is $(\Mstar)_a$.
  \item If $(\LU)_a$ satisfies the pushout product axiom \wrt the box product,
  then so does $(\Mstar)_a$.
\end{enumerate}
Assume in addition that all elements of $I$ are $h$-cofibrations in $\LU$ and
$\mathcal{W}$ is a class of equivalences detected by a family of fixed point
functors to spaces. Then:
\begin{enumerate}[a)] \setcounter{enumi}{6}
  \item Both $(\LU)_a$ and $(\Mstar)_a$ are left proper.
  \item Both $(\LU)_a$ and $(\Mstar)_a$ satisfy the unit axiom and
  monoid axiom \cite[Def.~3.3]{schwedeshipley:algsmods}. 
\end{enumerate}
\end{thm}


Before turning to the proof, we record some technical, but very useful results.

\begin{lemma}[\cite{schwede:GHTv1}, Cor.~A.30] \label{preservation_h-cofibrations}
Let $\mathcal{C, C'}$ be two cocomplete categories which are enriched and
tensored over spaces. Let $G \colon \mathcal{C} \to \mathcal{C'}$ be a
continuous functor that preserves pushouts along $h$-cofibrations and commutes
with taking tensors with the unit interval $I$. Then $G$ takes
$h$-cofibrations in $\mathcal{C}$ to $h$-cofibrations in $\mathcal{C'}$.
\end{lemma}


\begin{lemma}[Gluing lemma] \label{gluing_lemma}
Consider the following pushout diagram in $\LU$ or $\Mstar$, where one of the
maps $f$ or $g$ is an $h$-cofibration.
\begin{center}
$ \xymatrix@M=8pt{
	X \ar[r]^f \ar[d]_g & Y \ar[d]^h \\
	Z \ar[r]_k & W
}$
\end{center}
If $f$ is a global equivalence, then so is $k$. The statement remains true if
``global equivalence'' is replaced with any class of weak equivalences detected
by a family of fixed point functors to spaces.
\end{lemma}

\begin{proof}
Colimits in $\Mstar$ are created in $\LU$. Since one of the legs of the pushout is an $h$-cofibration, it is a closed embedding of
spaces. Thus, taking fixed points with respect to any closed subgroup of $\L(1)$ preserves the pushout. Moreover, by
Lemma~\ref{preservation_h-cofibrations}, taking fixed points sends $h$-cofibrations in $\LU$ (or $\Mstar$, respectively) to
$h$-cofibrations of spaces. Now the claim follows from the gluing lemma for $h$-cofibrations and weak homotopy
equivalences in spaces.
\end{proof}

The next observation is obtained by composing several standard adjunctions.

\begin{lemma} \label{F explicitly}
The underlying $\L$-space of $\Fbox{Y}{Z}$ is given by $\LU(Y, \LU(\L(2), Z))$ with actions as
follows: The space of $\L(1)$-equivariant maps $\LU(\L(2), Z)$ is formed \wrt the left $\L(1)$-action on $\L(2)$ induced by post-composition of linear
maps. This mapping space is an $\L$-space via the right $\L(1)$-action on $\L(2)$ induced by pre-composition on the second
summand of $(\R^\infty)^2$. Finally, the $\L(1)$-action on $\Fbox{Y}{Z}$ comes from the right $\L(1)$-action on $\L(2)$ induced by
pre-composition on the first summand of $(\R^\infty)^2$.
\end{lemma}

\begin{prop} \label{prop properties Fbox}
Let $Y$ be any $\L$-space and consider the functor $\Fbox{Y}{-} \colon \LU \to \LU$. 
\begin{enumerate}[i)]
    \item If $f$ is a closed embedding, then so are $\LU(Y,f)$ and $\Fbox{Y}{f}$.
    \item The functor $\LU(\L(2),-)$ takes sequential colimits along closed embeddings to
    sequential colimits along closed embeddings.
    \item The functor $\Fbox{*}{-}$ preserves sequential colimits along closed embeddings.
    \item If $\mathcal{W}$ is a class of weak equivalences satisfying the assumptions of
    Theorem~\ref{thm_A_strict}, then $\Fbox{*}{-}$ preserves and reflects $\mathcal{W}$.
\end{enumerate}
\end{prop}

\begin{proof}
\emph{Ad i):} The functor $\LU(Y,-)$ preserves closed embeddings because $\LU(Y,Z)$ is
topologized as a closed subspace of $\U(Y,Z)$. The functor $\Fbox{Y}{-}$ is a composition of
$\LU(\L(2),-)$ and $\LU(Y,-)$. \\
\emph{Ad ii):} Any choice of linear isometry $\R^\infty \cong (\R^\infty)^2$ induces an
isomorphism of $\L$-spaces $\L(2) \cong \L(1)$, thus the underlying space of $\LU(\L(2),Z)$ is
naturally isomorphic to $Z$. It follows that for any sequence of closed embeddings of $\L$-spaces
\[ Z_0 \to Z_1 \to Z_2 \to \ldots, \]
the canonical map
\[ \colim_i \LU(\L(2), Z_i) \to \LU(\L(2), \colim_i Z_i) \]
is a homeomorphism of spaces. Moreover, it is equivariant \wrt the $\L(1)$-action induced by
precomposition on the first summand of $(\R^\infty)^2$. \\
\emph{Ad iii):} By part ii) and Lemma~\ref{F explicitly}, it suffices to show that $\LU(*,-)$
preserves sequential colimits along sequences of closed embeddings. This is true because it is
just the fixed point functor $\LU(\L(1)/\L(1), -) \cong (-)^{\L(1)}$. \\
\emph{Ad iv):} Let $f \colon X \to Y$ be in $\mathcal{W}$. In the diagram of $\L$-spaces
\begin{center}
$ \xymatrix@M=10pt@C=3pc{
	X \ar[d]_{f} \ar[r]^{\bar{\lambda}} & \Fbox{*}{X} \ar[d]^{\Fbox{*}{f}} \\
	Y \ar[r]^{\bar{\lambda}} & \Fbox{*}{Y}
}$
\end{center}
both horizontal maps are strong global equivalences by Proposition~\ref{prop:lambda}. The strong
global equivalences are contained in the class of weak equivalences $\mathcal{W}$ by
assumption~iii) of Theorem~\ref{thm_A_strict}, thus $\Fbox{*}{f}$ is a weak equivalence if and
only if $f$ is..
\end{proof}

\begin{lemma} \label{smallness}
All $\L$-spaces, co-$*$-modules and $*$-modules are small \wrt sequences of
closed embeddings in the sense of \cite[Def.~2.1.3]{hovey:modelcats}.
\end{lemma}



\begin{proof}
The forgetful functors $\Mstar \to \LU$ and $\LU \to \U$ both have left adjoints, so colimits in either category can be formed in
$\U$. Consequently, all $\L$-spaces and $*$-modules are small \wrt sequences of closed embeddings. Colimits in $\Mcostar$ are
computed by applying $F = \Fbox{*}{-}$ to a colimit formed in $\LU$. By Proposition \ref{prop
properties Fbox}, $F$ preserves sequential colimits along closed embeddings, thus the smallness
statement for $\Mcostar$ follows from the one for $\LU$.
\end{proof}

In order to prove Theorem~\ref{thm_A_strict}, we construct an intermediate model
structure $(\Mcostar)_a$ on co-$*$-modules, thus exploiting the fact that, up
to equivalence of categories, $\Mstar$ is a category of algebras over a
well-behaved monad. This approach was used by Blumberg, Cohen and
Schlichtkrull to transport their non-equivariant model structure
in \cite[Sect.~4.6]{BCS:THH}, and goes back to \cite{EKMM}. Consider the
following diagram:

\numberwithin{equation}{section}
\setcounter{equation}{6}
\begin{equation} \label{lifting_diagram}
\xymatrix@M=10pt@C=4pc{
\LU \ar@<0.7ex>[r]^(.4){F} & \LU\left[
\mathbb{F} \right] \cong \Mcostar \ar@<0.7ex>[l]^(.6){R} \ar@<0.7ex>[r]^(.6){-
\Lbox *} & \Mstar \ar@<0.7ex>[l]^(.38){\Fbox{*}{-}} }
\end{equation}
\setcounter{definition}{7}

We have seen in Proposition \ref{mirror_image} that the adjunction on the right
hand side is an equivalence of categories. The proof of the identification $\LU\left[
\mathbb{F} \right] \cong \Mcostar$ is identical with the proof of
\cite[Prop.~II.2.7]{EKMM}, where $\mathbb{F}$ denotes the monad $\mathbb{F}
= RF$ associated to the adjunction on the left hand side.

The proof of Theorem~\ref{thm_A_strict} is built around a standard result which transports model structures along adjunctions and is
sometimes referred to as ``Kan's transfer theorem''. The formulation below is a slight variation of
\cite[Lemma~2.3]{schwedeshipley:algsmods}. Our condition (R3) is more general than that of Schwede-Shipley, but may be harder to
verify in general. In the case of interest in this paper, it comes for free.

\begin{thm}[Lifting of model structures] \label{lifting_thm}
Let $\mathcal{C}$ be a cofibrantly generated model category and $I$
(respectively $J$) a set of generating (acyclic) cofibrations. Let $T$ be a
monad on $\mathcal{C}$ and denote by $I_T$ and $J_T$ the images of $I$ and $J$,
respectively, under the free $T$-algebra functor. Assume that
\begin{enumerate}[(R1)]
    \item the domains of $I_T$ and $J_T$ are small relative to
    $I_T \mbox{-} cell$ and $J_T \mbox{-} cell$, respectively
    \item every morphism in $J_T \mbox{-} cell$ is sent to a weak equivalence in
    $\mathcal{C}$ under the forgetful functor
    \item the category $\mathcal{C}\left[ T \right]$ of $T$-algebras is
    cocomplete.
\end{enumerate}
Then $\mathcal{C}\left[ T \right]$ is a cofibrantly
generated model category with generating sets of (acyclic) cofibrations $I_T$
(respectively $J_T$), and weak equivalences and fibrations detected by the
forgetful functor to $\mathcal{C}$.
\end{thm}

\begin{cor} \label{cor_MS_Mcostar}
Given a model category $(\LU)_a$ as in Theorem \ref{thm_A_strict}, the category
of co-$*$-modules admits a cofibrantly generated model structure $(\Mcostar)_a$
with weak equivalences and fibrations detected by the forgetful functor $R
\colon \Mcostar \to (\LU)_a$. Sets of generating cofibrations and acyclic
cofibrations are given by $\Fbox{*}{I}$ and $\Fbox{*}{J}$, respectively.
\end{cor}

\begin{proof}
We verify the requirements of Theorem \ref{lifting_thm}. All colimits exist
since the forgetful functor to $\L$-spaces has a left adjoint. The
smallness statement is a special case of Lemma \ref{smallness}.
We now prove (R2): Let $j \colon A \to B$ be a morphism in $J$. Let $Y$ be the pushout of the left
hand square of co-$*$-modules and let $Y_0$ be the pushout in the right hand square of
$\L$-spaces:
\begin{center}
$ \xymatrix@M=8pt{
	F(A) \ar[r] \ar[d]_{F(j)} & X \ar[d]^g & \; &
	A \ar[r] \ar[d]_j & RX \ar[d]^{g_0} \\
	F(B) \ar[r] & Y & \; &
	B \ar[r] & Y_0
}$
\end{center}
Under the functor $F$, the right hand square is taken to the pushout square
\begin{center}
$ \xymatrix@M=10pt@C=3pc{
	F(A) \ar[r] \ar[d]_{F(j)} & (FR)(X) \cong X \ar[d]^{F(g_0)} \\
	F(B) \ar[r] & F(Y_0),
}$
\end{center}
but $(FR)(X) \cong X$, hence $F(Y_0) \cong Y$ by uniqueness of the pushout,
and the maps $g$ and $F(g_0)$ agree under this isomorphism. The map $j$ is an acyclic cofibration
and a closed embedding by assumption. Both of these properties are stable under cobase change,
hence $g_0$ is an acyclic cofibration and a closed embedding. Then by Proposition~\ref{prop
properties Fbox}, the map $g \cong F(g_0)$ is a closed embedding and a weak equivalence. \\
Finally, we claim that the collections of maps that are simultaneously closed
embeddings and weak equivalences is closed under transfinite composition in $\Mcostar$. This is
true in $\LU$ by assumption iv) of Theorem~\ref{thm_A_strict}, but colimits in $\Mcostar$ are not
constructed in $\LU$. More precisely, they are obtained by applying $F$ to a colimit formed in
$\LU$. Since $F$ preserves the class of weak equivalences by Proposition~\ref{prop properties
Fbox}, the claim follows. Altogether, we have shown that all relative $J_T$-cell complexes are
weak equivalences.
\end{proof}

We are now ready to give the

\begin{proof}[Proof of Theorem~\ref{thm_A_strict}]
The model structure $(\Mcostar)_a$ from Corollary \ref{cor_MS_Mcostar}
transports along the equivalence of categories
\[ \xymatrix@M=10pt@C=4pc{
\Mcostar \ar@<0.7ex>[r]^(.5){-
\Lbox *} & \Mstar \ar@<0.7ex>[l]^(.5){\Fbox{*}{-}} }
\]
to a model structure $(\Mstar)_a$ with weak equivalences and fibrations detected
by the composite $R \circ \Fbox{*}{-} \colon \Mstar \to \LU$, which proves
c).
Sets of generating (acyclic) cofibrations are given by the images of $I$ (resp.
$J$) under $\Fbox{*}{-} \Lbox * \colon \LU \to \Mstar$, which is
naturally equivalent to the functor $(-) \Lbox *$ by Proposition
\ref{mirror_image}, thus proving part~a).
Hypothesis iii) and Proposition \ref{prop:lambda} imply that $\Fbox{*}{-}$ preserves and
reflects the weak equivalences $\mathcal{W}$; now b) follows immediately. In order to show d),
it suffices to show that the left hand adjunction in (\ref{lifting_diagram}) is a Quillen
equivalence. It is a Quillen adjunction by construction. It is a Quillen equivalence
because $RF$ preserves and reflects weak equivalences and because the unit $\bar{\lambda} \colon
X \to RF(X)$ is a strong global equivalence, see Proposition~\ref{prop properties Fbox} and
Proposition~\ref{prop:lambda}, respectively. \\
Now we proof the enhancements e) through h): \\
\emph{Ad e):} Assume that $(\LU)_a$ is right proper. Then so is $(\Mstar)_a$
since the right adjoint $R \circ \Fbox{*}{-} \colon \Mstar \to \LU$ preserves
pullbacks, and preserves and reflects weak equivalences and fibrations. \\
Now assume that $(\LU)_a$ is topological. Let $f \colon X \to Y$ be a generating
cofibration for $(\LU)_a$ and $i \colon A \to B$ any cofibration in $\U$. By
assumption, the pushout product \[ f \Box i \colon P = Y \times A \cup_{X \times
A} X \times B \to Y \times B \] is again a cofibration in $(\LU)_a$. The map $f
\Lbox *$ is a generating cofibration in $(\Mstar)_a$ whose pushout product with
$i$ is isomorphic to
\[ (f \Box i) \Lbox * \colon P \Lbox * \to (Y \times B) \Lbox *.\]
As $- \Lbox * \colon (\LU)_a \to (\Mstar)_a$ is a left Quillen functor, this map
is a cofibration in $\Mstar$. If $f$ is a generating acyclic cofibration or $i$
any acyclic cofibration, then $f \Box i$ is an acyclic cofibration in $\LU$,
hence so is $(f \Lbox *) \Box i \cong (f \Box i) \Lbox *$ in $\Mstar$.
\\
\emph{Ad f):} There are natural isomorphisms
\[ (X \Lbox *) \Lbox (X' \Lbox *) \cong (X \Lbox X') \Lbox * \]
for all $\L$-spaces $X$ and $X'$. Similar reasoning as in the proof of g)
then shows that for two generating cofibrations $f \colon A
\to B$ and $f' \colon A' \to B'$ for $(\LU)_a$, the pushout product of $f \Lbox
*$ and $f' \Lbox *$ is isomorphic to $(f \Box f') \Lbox *$, hence is a
cofibration in $\Mstar$, and acyclic if $f$ or $f'$ is a generating acyclic
cofibration. \\
\emph{Ad g):} Left properness follows immediately from Lemma \ref{gluing_lemma}.
\\
\emph{Ad h):} The box product is weakly equivalent to the categorical product
by Proposition~\ref{prop:lambda} and the assumption that any strong global
equivalence is a weak equivalence in $(\LU)_a$. As the weak equivalences are detected by
fixed point functors, the functor $(-) \Lbox Z$ preserves weak equivalences,
where $Z \in \LU$ is any $\L$-space. The unit axiom follows immediately. \\
Let $\mathcal{A}$ denote the class of morphisms $j \Lbox Z$ where
$j$ is an acyclic cofibration and $Z \in \LU$ is arbitrary. All
cofibrations in $(\LU)_a$ are $h$-cofibrations. As just observed, the functor
$(-) \Lbox Z$ preserves weak equivalences. Because of
Lemma \ref{preservation_h-cofibrations}, it always preserves $h$-cofibrations,
too. Moreover, the class of weak equivalences which are $h$-cofibrations is
stable under cobase changes (by Lemma \ref{gluing_lemma}), transfinite composition, and
retracts. Thus, all relative $\mathcal{A}$-complexes are weak equivalences. \\
The same proof applies to $(\Mstar)_a$.
\end{proof}

\phantomsection
\addcontentsline{toc}{section}{Bibliography}
\bibliographystyle{amsplain}
{\footnotesize \bibliography{topology} }

\end{document}